\documentclass[11pt]{amsart}
\usepackage[margin=1.5in]{geometry}

\usepackage{adjustbox}
\usepackage{float}

\date{\today}

\usepackage{amsmath}
\usepackage{amsfonts}
\usepackage{amssymb}
\usepackage{amsthm}
\usepackage{enumerate}
\usepackage{color}
\usepackage{stmaryrd}
\usepackage{slashed}
\usepackage{relsize}
\usepackage{geometry}
\usepackage{tikz-cd}
\usepackage{mathtools}
\usepackage{xcolor}
\definecolor{darkgreen}{rgb}{0,0.50,0} 
\definecolor{darkred}{rgb}{0.55,0,0}
\definecolor{darkblue}{rgb}{0,0,0.6}
\definecolor{darkteal}{rgb}{0.0, 0.25, 0.5}

\usepackage{caption}
\usepackage{enumitem}

\definecolor{color1}{RGB}{12, 12, 149}
\definecolor{color2}{RGB}{12, 32, 161}
\definecolor{color3}{RGB}{9, 146, 212}
\definecolor{color4}{RGB}{8, 184, 226}
\definecolor{color5}{RGB}{28, 253, 175}
\definecolor{color6}{RGB}{39, 255, 151}

\usetikzlibrary{automata,arrows,calc,positioning}

\usepackage{array}

	\definecolor{darkgreen}{rgb}{0,0.50,0} 
	\definecolor{darkred}{rgb}{0.55,0,0}
	\definecolor{darkblue}{rgb}{0,0,0.6}

	\usepackage[pdfborder=0,pagebackref,colorlinks,citecolor=darkgreen,linkcolor=darkred,urlcolor=darkblue]{hyperref}
	\usepackage{cleveref}

	\newtheorem{thm}{Theorem}[section]
	\newtheorem{theorem}[thm]{Theorem}
	\newtheorem{corollary}[thm]{Corollary}
	
	\newtheorem{proposition}[thm]{Proposition}
	\newtheorem{lem[thm]}{Lemma}
	\newtheorem{lemma}[thm]{Lemma}
    \newtheorem{notn}[thm]{Notation}
	\theoremstyle{definition}
	\newtheorem{definition}[thm]{Definition}

	\newtheorem{example}[thm]{Example}

    	\newtheorem{question}[thm]{Question}
	
	\newtheorem{mnthm}[thm]{Theorem}
	
	\crefname{mnthm}{Theorem}{Theorems}

	\usepackage{verbatim}
	
	\usepackage[normalem]{ulem}
	
	\newcommand{\Z}{\mathbb{Z}}

	\newcommand{\R}{\mathbb{R}}

	\newcommand{\Pin}{\operatorname{Pin}}
	\newcommand{\rank}{\operatorname{rank}}
	\newcommand{\Sq}{\operatorname{Sq}}
	
	\newcommand{\colim}{\operatorname{colim}}
	
	\newcommand{\kerv}{\operatorname{kerv}}

	\newcommand{\SKK}{\mathrm{SKK}}

	\newcommand{\SKKxi}{\mathrm{SKK}^\xi}
	\newcommand{\RP}{\mathbb{RP}}
	
	\newcommand{\HP}{\mathbb{HP}}

	\DeclareMathOperator\BO{BO}

	\usepackage{booktabs}
	
	\usepackage{bm}
	
	\author[R. S. Hoekzema]{Renee  S. Hoekzema}
	\address{Department of Mathematics, Vrije Universiteit Amsterdam, Netherlands}
	\email{r.s.hoekzema@vu.nl}
	\author[L. Stehouwer]{Luuk Stehouwer}
	\address{Durham University, United Kingdom}
	\email{stehouwer.d.luuk@durham.ac.uk}
	\author[S. Vesel\'a]{Simona Vesel\'a}
	\address{Max Planck Institute for Mathematics, Germany}
	\email{vesela@math.uni-bonn.de}

	\title{The Euler characteristic of $\Pin^{\pm}$ manifolds}

\begin{document}

		\begin{abstract} 
			We completely determine the possible parities of the Euler characteristic of closed $\Pin^{\pm}$-manifolds of a given dimension. As an application we partially calculate the controlled cut-and-paste groups $\SKK_{2k+1}^{\Pin^{\pm}}$, which were previously undetermined. 
		\end{abstract}
		
		\maketitle
		
		\tableofcontents

		\section{Introduction}

		In every even dimension there exist closed manifolds of every possible Euler characteristic. Orientable closed manifolds, however, are restricted to having even Euler characteristic in dimensions $2 \pmod 4$. Similarly imposing other orientable tangential structures such as $BSpin,BString,\cdots$ restricts the dimensions in which the Euler characteristic can be odd further, see \cite{ReneeEulerChar}. What restrictions on the Euler characteristic are imposed by different tangential structures is not known in general. This question is only of interest in even dimensions, as odd-dimensional closed manifolds have Euler characteristic zero.

		\subsection{$\Pin^+$, $\Pin^-$ structures.}
		 Pin groups are the non-orientable analogues of the Spin groups.

The tangential structure $B\Pin^+ \to \BO$ can be defined as the fiber of $w_2 \colon \BO \to B^2\Z/2$ and $B\Pin^-$ as the fiber of $w_2 + w_1^2 \colon \BO \to B^2\Z/2$, see \Cref{subsec:TangentialStr}. They are more classically defined using the $\colim_{n\to\infty}\Pin^{\pm}_n$ of the $\Pin^{\pm}_n$- groups, which are in turn defined using Clifford algebras~\cite[Chapter I]{MR1031992}. In recent years $\Pin^{\pm}$ groups have gotten more attention in the physics literature due to their application to time-reversal symmetric systems~\cite{MFCV14, KTTW15, Witten16parity, KPT20, freedhopkins}.

                Our main result is the following.
		
		\begin{mnthm}\label{mnthm:Mainthm}
			Let $\xi$ be the $\Pin^+$ or the $\Pin^-$ structure and let $n$ be an even integer. Then the statement
			\[\text{All }n\text{-dimensional } \xi\text{-manifolds have even Euler characteristic.}\] is true if and only if  $\xi=\Pin^+$ and $n\equiv 2,6\pmod{8}$ or, $\xi=\Pin^-$ and $n\equiv 4\pmod{8}$.
		\end{mnthm}
        
        In dimensions where the Euler characteristic is restricted to be even, this theorem is a consequence of \Cref{thm:8kp2PinPlus}, \Cref{prop:PinMinusEulerChar} and \Cref{prop:pinplusdim6}.
In \Cref{sec:ProjSpaces}, we provide examples of closed $\Pin^{\pm}$-manifolds with odd Euler characteristic in all remaining dimensions. 
We show that $n$-dimensional $\Pin^{\pm}$-manifolds have even Euler characteristic by showing that admitting a $\Pin^{\pm}$ structure implies that the top Stiefel-Whitney class $w_n$ vanishes.
For $\Pin^{+}$ in dimensions $6 \pmod{8}$ and $\Pin^{-}$ in dimensions $4 \pmod{8}$, we do this by showing that the top dimensional Wu class $v_{n/2}$ vanishes, which implies that $w_n=v_{n/2}^2$ vanishes as well. For $\Pin^{+}$ in dimensions $2 \pmod{8}$, the top Wu class need not vanish. For example, the Klein bottle has nonzero $v_1=w_1$ but $w_2=0$, see also \Cref{lem:Pin+counterexamples}.
Nonetheless we obtain, using a more complicated calculation, that $w_n$ vanishes in these dimensions. This is done using a theorem of Dold, see \Cref{thm:Dold}.

		The following is a table summarising the results of \Cref{mnthm:Mainthm}, with examples of manifolds with odd Euler characteristic in the admissible dimensions.
        
		\begin{table}[h!]
			\begin{tabular}{|l|l | l|}
				\hline
				dim & $\Pin^-$ &$\Pin^+$ 
				\\\hline 
				$8k$& 
				$\mathbb{HP}^{2k}$ & $\mathbb{RP}^{8k}$, $\mathbb{HP}^{2k}$
				\\
				$8k+2$& 
				$\mathbb{RP}^{8k+2}$  & - (Thm. \ref{thm:8kp2PinPlus})
				\\
				$8k+4$& 
				- (Cor. \ref{prop:PinMinusEulerChar}) & $\mathbb{RP}^{8k+4}$
				\\
				$8k+6$& 
				$\mathbb{RP}^{8k+6}$  & - (Lem. \ref{prop:pinplusdim6})
				\\
				\hline
			\end{tabular}
			\caption{Examples of $\Pin^{\pm}$-manifolds with odd Euler characteristic in different dimensions, with dashes indicating that no such manifolds exist.}
			\label{table:StrucutresOfProjSpaces}
		\end{table}
	
	\subsection{Cut and paste groups}\label{subsec:SKKIntro}
		Our results have implications for controlled cut-and-paste groups, also known as SKK groups, which we studied in detail in \cite{HSV2025}. The same groups are studied in many other places in the literature, e.g.  \cite{MR153021, skbook, bokstedtsvane, schommerpriesinvertible, MR4565254, KST}. SKK groups are of particular interest to physics because they classify non-unitary topological quantum field theories \cite{KST,  HSV2025}.
        
		Let $\xi\colon B\to \BO$ be a fibration and let a $\xi$-manifold be a manifold together with a lift of the classifying map of its tangent bundle along $\xi$. Let $n\geq 0$ be an integer.
		
		\begin{definition}
			The group $\SKK_n^{\xi}$ is defined by quotienting the monoid of $n$-dimensional $\xi$-manifolds by the relation
			\[M_1\cup M_3+ M_2\cup M_4\sim_{\SKK} M_1\cup M_4+ M_2\cup M_3,\]
			where $M_1, M_2, M_3, M_4$ are compact manifolds the identification of their boundaries $\partial M_1\cong \partial M_2\cong \partial M_3\cong \partial M_4$.
		\end{definition}

        As a consequence of \Cref{mnthm:Mainthm} and the work in \cite{HSV2025} we obtain the following result for SKK groups of $\Pin^\pm$-manifolds in odd dimensions, see \Cref{cor:SKKPin}:
        \begin{mnthm}
For $n\geq 1$ an odd integer, we have
			\begin{eqnarray*}
			\SKK_n^{Pin^-} &\cong&     \begin{cases}
            \Omega_{n}^{\Pin^-}	&\text{ if } n=1,5,7 \pmod{8}\\
				\Omega_{n}^{\Pin^-}\oplus\Z/2& \text{ if }  n=3 \pmod{8}
			\end{cases}\\
\SKK_n^{Pin^+} &\cong& \begin{cases}
				\Omega_{n}^{\Pin^+}&\text{ if }n=3,7 \pmod{8}\\
				\Omega_{n}^{\Pin^+}\oplus\Z/2& \text{ if }n=5 \pmod{8} 
\end{cases}
            \end{eqnarray*}
            and lastly for $\Pin^+$-manifolds in dimension $n= 1 \pmod{8}$ there is a short exact sequence $\Z/2\to\SKK_n^{\Pin^+}\to \Omega_n^{\Pin^+}$.
        \end{mnthm}

\subsection{Conventions} All manifolds are smooth and  compact. All structures on manifolds are tangential i.e.\ for $\xi\colon B\to \BO$, a manifold is a $\xi$-manifold if it is endowed with a $\xi$-structure on its tangent bundle. We note that in surgery theory {\it normal structures} are more common and in that case a tangential $\Pin^{+}$-structure corresponds to a normal $\Pin^-$-structure and the other way around.

\subsection*{Acknowledgements}
The authors would like to thank Matthias Kreck, Stephan Stolz and Peter Teichner for many discussions about their unpublished work \cite{KST}. SV would like to thank Viktor Vesel\'y for his help with programming.
RH is supported by the Dutch Research Council (NWO) through the grant
OCENW.XL.23.015. LS is supported by the ERC Consolidator Grant “SYMSPEC”. SV is supported by the Max Planck Institute for Mathematics in Bonn.

\subsection*{Statement on the use of artificial intelligence and programming}
No portion of this work was conceived of or written by artificial intelligence. A python code computing Dold and Steenrod relations was used to come up with the proof of \Cref{thm:8kp2PinPlus}.

		\section{Preliminaries}

        \subsection{Characteristic classes of manifolds}
		Let $O(n)$ be the group of orthogonal transformations of $\R^n$. Under the natural stabilisation map $O(n)\to O(n+1)$ we can define the colimit $
		O=\colim_{n\to\infty}O(n)$. Applying the bar construction we obtain a space $\BO$, which is the classifying space of stable vector bundles. The $\Z/2$-cohomology ring of $\BO$ is well known to be 
		\[H^*(\BO;\Z/2)\cong \Z/2[w_1,w_2,w_3,\cdots],\]
		where $w_i\in H^i(\BO;\Z/2)$ is the \emph{$i$-th Stiefel--Whitney class}.
        While free as generators of the cohomology ring, the Stiefel-Whitney classes are related by the action of the Steenrod algebra under the following formula, which follows from the Cartan formula and the splitting principle \cite{MilnorStasheff1974}
				\begin{equation}\label{eqn:Sqw}
			\Sq^i w_j = \sum_{k=0}^i {j-i + k -1 \choose k} w_{i-k} w_{j+k}.
		\end{equation}

		Let $M^n$ be a manifold. The classifying space of its tangent bundle gives a map $TM\colon M\to \BO(n) \to \BO$. Define $w_i(M)$ to be the pullback of $w_i$.         
        The \emph{Wu classes} $v_i(M)\in H^i(M;\Z/2)$ for $i=0,1,\cdots$ of $M^n$ are uniquely defined by the property that $\Sq^k(x_{n-k}) = v_k(M) x_{n-k}$ for all $x_{n-k}\in H^{n-k}(M;\Z/2)$. We have that $v_i(M)=0$ for $i>\frac{n}{2}$.

		The Wu classes of a manifold are related to its Stiefel--Whitney classes through the Wu formula $w = \Sq(v)$, or
        	\begin{equation}\label{eq:Wuformula}
w_1(M)+ \cdots + w_n(M)  = \left(\Sq^1 + \cdots + \Sq^{\lfloor\frac{n}{2}\rfloor} \right)(v_1(M) + \cdots + v_{\lfloor\frac{n}{2}\rfloor}(M) ).
		\end{equation} 
        This gives rise to a recursive formula of Wu classes in terms of the Stiefel--Whitney classes 
		\begin{equation}\label{eq:RecursiveFormula}
			v_i(M) = w_i(M) + \sum_{k=1}^{{\lfloor}i/2{\rfloor}} \Sq^{k} v_{i-k}(M).
		\end{equation}

		\begin{example}\label{ex:SmallWu}
		The first few Wu classes are
			\begin{align*}
				&v_1=w_1\\
				&v_2=w_2+w_1^2\\
				&v_3=w_3+\Sq^1(v_2)=w_3+\Sq^1(w_2)=w_2w_1.
			\end{align*}
		\end{example}

		\begin{lemma}\label{lem:EquatioWuWn}
			Let $M^{n}$ be a manifold of dimension $n=2k$. Then $w_n(M)=v_k(M)^2$. 
		\end{lemma}
		\begin{proof}
			This follows directly from the Wu formula \cref{eq:Wuformula}.
		\end{proof}

		If $M^{2k}$ is a manifold, we will call the class $v_k(M)$ its \emph{top Wu class}.
		
		We can use formula (\ref{eq:RecursiveFormula}) to define \textit{universal} Wu classes in the cohomology of BO, $v_k \in H^k(\BO;\Z/2)$. Note that the formula 
                \begin{equation}
        \label{eq:Sqdoldrel}
            \Sq^k(w_{n-k}) = v_k x_{n-k}
        \end{equation} is in general not true universally.
        Indeed, filling in any degree $n-k$ polynomial in Stiefel-Whitney classes for $x_{n-k}$, we obtain a relation between characteristic classes that holds for any $n$-dimensional manifold, while it is not necessarily true in the cohomology of $BO$.
        It was shown by Dold that all relations between Stiefel--Whitney classes in $H^n(\BO;\Z/2)$ that hold for all $n$-dimensional manifolds are generated by precisely relations like \eqref{eq:Sqdoldrel}:
		\begin{theorem}\cite[{Satz 1}]{dold1956vollstandigkeit}\label{thm:Dold}
			The elements 
			\begin{equation}
				\Sq^{n-k}(p)-v_{n-k}\;p\in H^n(\BO;\Z/2)
				\label{eqn:Dold}
			\end{equation}
			for all elements $p\in H^k(\BO;\Z/2)$ generate all relations between Stiefel--Whitney numbers of $n$-dimensional manifolds.
		\end{theorem}

		We can restate \Cref{thm:Dold} to say that for any manifold $M^n$ with tangent bundle $TM\colon M\to \BO$ the induced map $H^n(\BO;\Z/2)\to H^n(M;\Z/2)$ factors through the quotient group homomorphism $H^n(\BO;\Z/2)\to H^n(\BO;\Z/2)/I_n$, where $I_n$ is a subgroup of $H^n(\BO;\Z/2)$ generated by the set
		\[
		\left\{Sq^{n-k}(p)+v_{n-k}\;p \; |\;  p \in H^k(\BO;\Z/2) \right\}.
		\]
       
		In other words, $H^n(\BO;\Z/2)/I_n$ is the  universal object with respect to the property that for every manifold $M^n$, the map $TM^{*}\colon H^n(\BO;\Z/2)\to H^n(M;\Z/2)$ factors through it.

The top-dimensional Stiefel-Whitney class relates to the parity of the Euler characteristic via the following classical lemma.        
		\begin{lemma}\label{lem:EulerCharTopSWClass}
			Let $TM\colon M\to \BO$ be the stable tangent bundle of some closed manifold $M$. Then $\chi(M)\equiv \langle w_n(M),[M]\rangle\pmod{2}$ where $[M]\in H_n(M;\Z/2)$ is the $\Z/2$-fundamental class of manifold $M$. 
		\end{lemma}
In particular, if $w_n(M)=0$ then $\chi(M)$ is even.
		\begin{proposition}\label{prop:ConditionEvenEuler}
			Let $\xi\colon B\to \BO$ be any tangential structure. If $w_n\in\xi^*(I_n)\subset H^n(B;\Z/2)$ then all closed $\xi$-manifolds $M^n$ have even Euler characteristic.
		\end{proposition}
		\begin{proof}
			In this case $w_n(M)$ is trivial for any $n$-dimensional $\xi$-manifold $M$ by  \Cref{thm:Dold}, so by \Cref{lem:EulerCharTopSWClass} $M$ has even Euler characteristic. 
		\end{proof}
        
		\subsection{Tangential structures as homotopy fibers of cohomology classes}\label{subsec:TangentialStr}
        
		Let $w\in H^i(\BO;\Z/2)$ be a homogeneous cohomology class. We can define $\xi_w\colon B_{w}\to \BO$ to be the homotopy fibre of the map $\BO\xrightarrow{w}K(\Z/2;i)$, so that a manifold $M$ has a tangential $\xi_w$-structure if and only if $w(M)=0$.
		
		\begin{example} 
			We have 
			\begin{enumerate}
            \item $\xi_{w_1} \colon B_{w_1}\to \BO$ is the $BSO$ structure, i.e.\ an orientation
				\item $\xi_{w_2} \colon B_{w_2}\to \BO$ is the $B\Pin^{+}$ structure, $\xi_{\Pin^+}:=\xi_{w_2}$
				\item $\xi_{w_2 + w_1^2} \colon B_{w_2+w_1^{2}} \to \BO$ is the $B\Pin^{-}$ structure, $\xi_{\Pin^-}:=\xi_{w_2+w_1^2}$.
			\end{enumerate}
		\end{example}

\begin{notn}
    Let $w\in H^i(\BO;\Z/2)$ be a homogeneous cohomology class.
    We denote the ideal generated by actions of the Steenrod algebra $\mathcal{A}$ on $w$ by $(\mathcal{A}w)$. By $(\mathcal{A}w)_n$ we denote its elements in degree $n$, i.e.\ $(\mathcal{A}w)_n = (\mathcal{A}w) \cap H^n(\BO;\Z/2)$.
\end{notn}
Since $\xi_n^*(w)=0$ by definition of $\xi_w$, and the Steenrod squares are natural, it follows that $\xi_w^*(\mathcal{A}w)=0$, and hence we have: 
		\begin{lemma}\label{lem:betterSets}
			Let $w\in H^i(\BO;\Z/2)$ be a homogeneous cohomology class and let $\xi_{w}\colon B_w\to \BO$ be as above. Then  $ \xi_w^*(\langle I_n\cup (\mathcal{A}w)_n \rangle)=\xi_w^*(I_n)$, where $\langle I_n\cup (\mathcal{A}w)_n \rangle \subseteq H^n(BO; \Z/2)$ is the additive subgroup generated by $I_n$ and $(\mathcal{A}w)_n$.
		\end{lemma}
		
Hence, if we want to check whether a characteristic class $x\in H^n(\BO;\Z/2)$ vanishes on all $\xi_w$-manifolds, we can check whether $\xi_w^*(x) \in \xi_w^*(I_n) = \xi_w^*(\langle I_n\cup (\mathcal{A}w)_n\rangle)$, which follows if $x \in \langle I_n\cup (\mathcal{A}w)_n\rangle$, or equivalently $x \in  I_n \pmod{ (\mathcal{A}w)_n)}$.

		\section{Examples of $\Pin^{\pm}$-manifolds with odd Euler characteristic}\label{sec:ProjSpaces}
		
		In this section we verify the examples with odd Euler characteristics in Table \ref{table:StrucutresOfProjSpaces}. The total Stiefel--Whitney class of $\RP^n$ is given as  $w(\mathbb{RP}^{n}) = (1+t)^{n+1}$, for $t$ the degree 1 generator of $H^*(\mathbb{RP}^{n}; \Z/2) \cong \Z/2 [t]/t^{n+1}$. We obtain the table below.
\begin{table}[H]
			\begin{tabular}{lcccccccc}
				\hline
				& odd $\chi$ & $w_1$ & $w_2$ & $v_2$  & $\Pin^-$ ($v_2 = 0$) & $\Pin^+$ ($w_2 = 0$)  
				\\ \hline
				$\mathbb{RP}^{4k}$ & Y & $t$ & 0 & $t^2$      & N &  Y  \\
				$\mathbb{RP}^{4k+1}$ & N & 0 & $t^2$ & $t^2$      & N &  N 
				\\
				$\mathbb{RP}^{4k+2}$ & Y & $t$ & $t^2$ & 0      & Y &  N
				\\
				$\mathbb{RP}^{4k+3}$ & N & 0 & 0 & 0 &  Y &Y
				\\
				\hline
			\end{tabular}
			\caption{Real projective spaces admitting tangential $Pin^{\pm}$-structures. The label Y means yes, and N means no.}
			\label{tab:pinRPn}
		\end{table}

        Similarly there are quaternionic projective spaces $\HP^{n}$ of dimension $4n$. They can be endowed with a CW structure with one cell in each dimension divisible by $4$ up to and including $4n$. As such $\chi(\HP^{n})=n+1$ and they admit both $\Pin^+$ and $\Pin^-$-structures. Thus for even $n$ the manifold $\HP^{n}$ is a $\Pin^{\pm}$ with odd Euler characteristic.

		\section{Even Euler characteristic}

In this section we establish the results restricting the parity of the Euler characteristic of $\Pin^{\pm}$ manifolds.

		\subsection{\texorpdfstring{$\Pin^+$}{Pin+} manifolds in dimension \texorpdfstring{$8k+2$}{8k+2}}
		
		In this subsection we show
		
		\begin{theorem}\label{thm:8kp2PinPlus}
			Every $\Pin^+$-manifold in dimension $8k+2$ has even Euler characteristic.
		\end{theorem}

		\Cref{thm:8kp2PinPlus} will follow from the following proposition.
		\begin{proposition}\label{prop:8kp2ToPSWClass}
			Let $n=8k+2$. The universal Stiefel-Whitney class $w_n$ is an element of $\langle I_n\cup (\mathcal{A}w_2)_n\rangle$.
		\end{proposition}
		\begin{proof}[Proof of \Cref{thm:8kp2PinPlus} using \Cref{prop:8kp2ToPSWClass}]
			Since $w_n\in\langle I_n\cup (\mathcal{A}w_2)_n\rangle$ we can use \Cref{lem:betterSets} to get that $\xi_{\Pin^+}^*(w_n)\in \xi_{\Pin^+}^*(I_n)$. That implies that all $\xi_{\Pin^+}$-manifolds $M$ have $w_n(M)=0$  which suffices to finish the proof using \Cref{prop:ConditionEvenEuler}.
		\end{proof}

	We first prove a couple of technical lemmata.
		
		\begin{lemma}\label{lem:SqSWClasses}
			The following formulas hold in $H^*(\BO;\Z/2)$.
			\begin{enumerate}
				\item $\Sq^1(w_1^{2i})=0$
				\item $\Sq^2(w_1^{2i})=
				\begin{cases}
					w_1^{2i+2} & 2i\equiv 2,6\pmod{8} \\
					0 & 2i\equiv 0,4\pmod{8}
				\end{cases}$
				\item $\Sq^3(w_1^{2i})=0$
				\item $\Sq^4(w_1^{2i})=\begin{cases}
					w_1^{2i+4} & 2i\equiv 4, 
					6\pmod{8},\\
					0 & 2i \equiv 0,2 \pmod{8}.
				\end{cases}$
			\end{enumerate}
		\end{lemma}
		\begin{proof}
			The first formula follows from \cref{eqn:Sqw} as
			\[
			\Sq^1(w_1^{2i}) = w_1^i \Sq^1(w_1^i) + \Sq^1(w_1^i) w_1^i = 0.
			\]

			The second formula follows by the first formula using induction on $i$ for
			\begin{align*}
				\Sq^2(w_1^{2i}) &= w_1^2 \Sq^2(w_1^{2i-2}) + \Sq^1(w_1^2) \Sq^1( w_1^{2i-2}) +  \Sq^2(w_1^2) w_1^{2i-2} 
				\\
				&= w_1^2 \Sq^2(w_1^{2i-2}) + w_1^{2i+2}
			\end{align*}
			since $\Sq^2(w_1^2) = w_1^4$.
			
			The third formula follows since
			\[
			\Sq^3(w_1^{2i}) = 2 w_1^i \Sq^3(w_1^{i}) + 2 \Sq^1(w_1^i) \Sq^2(w_1^{i}) = 0.
			\]

			For the final formula, we use $\Sq^4(w_1^2) =  0$, the second formula, and induction after computing
			\begin{align*}
				\Sq^4(w_1^{2i}) &= \Sq^4(w_1^2) w_1^{2i-2} + \Sq^2(w_1^2) \Sq^2(w_1^{2i-2}) + w_1^2 \Sq^4(w_1^{2i-2})
				\\
				&= w_1^4 \Sq^2(w_1^{2i-2}) + w_1^2 \Sq^4(w_1^{2i-2}).
			\end{align*}
			We obtain that $\Sq^4(w_1^{2i}) = 0$ if $2i \equiv 0,2 \mod 8$ and $w_1^{2i+4}$ when $2i \equiv 4,6 \mod 8$.
		\end{proof}
		
		\begin{lemma}
			\label{it:4sqprop} Let $i,m$ be nonnegative integers such that $2m+2i\equiv 6\pmod{8}$ then modulo $(\mathcal{A}w_2)$ we have
			\[\Sq^4(w_1^{2i}w_{2m})=w_1^{2i}w_4w_{2m}+
			\begin{cases}
				w_1^{2i}w_{2m+4} & 2i\equiv 0\text{ (mod 8)}\\
				w_1^{2i+2}w_{2m+2} & 2i\equiv 2\text{ (mod 8)}\\
				w_1^{2i+4}w_{2m} & 2i\equiv 4\text{ (mod 8)}\\
				w_1^{2i}w_{2m+4}+w_1^{2i+2}w_{2m+2}+ w_1^{2i+4}w_{2m} & 2i\equiv 6\text{ (mod 8).}
			\end{cases}\]
		\end{lemma}
		\begin{proof}
			
			If $2m=2$ then $2i\equiv 4\pmod{8}$ and by the Cartan formula $\Sq^4(w_1^{2i}w_2)\in(\mathcal{A}w_2)$, but in this case also $w_1^{2i}w_4w_{2}+w_1^{2i+4}w_{2}\in (\mathcal{A}w_2)$, hence the formula holds.

			Next we assume $2m\geq 4$. We will make the following calculation in $H^*(\BO;\Z/2)$ 
			mod $(\mathcal{A}w_2)$, using \Cref{lem:SqSWClasses}, the Cartan formula, \cref{eqn:Sqw} and \cref{tab:binom} of binomial coefficients.
			
			\renewcommand{\arraystretch}{1.4}
			\begin{table}[h!]
				\centering
				\begin{tabular}{c|c|c|c|c|c|c|c|c}
					$i \pmod8=$ & $0$ & $1$ & $2$ & $3$ & $4$ & $5$ & $6$ & $7$ \\ \hline 
					$i\choose 0$&1&1&1&1&1&1&1&1\\ \hline
					$i\choose 1$&0&1&0&1&0&1&0&1\\ \hline
					$i \choose 2$ &$0$&$0$&$1$&$1$&$0$&$0$&$1$&$1$  \\ \hline
					$i \choose 3$ &$0$&$0$&$0$&$1$&$0$&$0$&$0$&$1$ \\ \hline
					$i \choose 4$ &$0$&$0$&$0$&$0$&$1$&$1$&$1$&$1$
				\end{tabular}
				\caption{For convenience of the reader, we collected some binomial coefficients modulo 2. They only depend on $i$ modulo $8$. This follows immediately from Lucas' theorem.}
				\label{tab:binom}
			\end{table}

						\begin{align*}
				&\Sq^4(w_1^{2i}w_{2m})\hspace{.5cm}=\hspace{.5cm}w_1^{2i}\Sq^4(w_{2m})+\Sq^2(w_1^{2i})\Sq^2(w_{2m})+\Sq^4(w_1^{2i})w_{2m}\\
				& = w_1^{2i}\sum_{k=0}^4 {{2m+k-5}\choose{k}}w_{4-k} w_{2m+k}
				+ \Sq^2(w_1^{2i})\sum_{k=0}^2 {{2m+k-3}\choose{k}}w_{2-k} w_{2m+k}
				+\Sq^4(w_1^{2i})w_{2m}\\
				&\stackrel{(w_2\equiv0)}{=}\left(w_1^{2i}w_4w_{2m}+{{2m-1}\choose{4}}w_1^{2i}w_{2m+4}\right)+\Sq^2(w_1^{2i}){{2m-1}\choose{2}}w_{2m+2}+\Sq^4(w_1^{2i})w_{2m} \\
				& = w_1^{2i}w_4w_{2m}+
				\left\{\!\begin{aligned}
					&w_1^{2i}w_{2m+4} &\text{ if }  2m\equiv 0,6\text{ (mod 8)} \\
					&0 &\text{ if } 2m\equiv 2,4\text{ (mod 8)} 
				\end{aligned}\right\}\\
				&
				+
				\left\{\!\begin{aligned}
					&w_1^{2i+2}w_{2m+2} &\text{ if }  2m\equiv 0,4\text{ (mod 8)} \text{ and } i \text{ odd}\\
					&0 &\text{ if } 2m\equiv 2,6 \text{ (mod 8)}\;\, \text{ or } i \text{ even}  \end{aligned}\right\}
				+ 
				\left\{\!\begin{aligned}
					&w_1^{2i+4}w_{2m} &\text{ if }  2i\equiv 4,6\text{ (mod 8)}\\
					&0 &\text{ if } 2i\equiv 0,2\text{ (mod 8)}
				\end{aligned}\right\}.
			\end{align*}
			Here we used that $w_2\in (\mathcal{A}w_2)$ and that ${d\choose 1}=0$ and ${d\choose 3}=0 \pmod{2}$ for even $d$.

			Filling in the condition $2m+2i\equiv 6 \pmod 8$ gives the desired formula.
			
		\end{proof}

		Fix a nonnegative integer $k$. For nonnegative integers $i$, define the elements
		\begin{equation}
			\mathfrak{r}_{2i} = 
			\begin{cases}
				w_1^{2i}w_{8k+2-2i}+w_1^{2i+4}w_{8k-2-2i} &  i\equiv 0\pmod{4}
				\\
				w_1^{2i+2}w_{8k-2i}+w_1^{2i+4}w_{8k-2-2i} &  i\equiv 1\pmod{4}
				\\
				0 & i\equiv 2 \pmod{4}
				\\
				w_1^{2i}w_{8k+2-2i}+w_1^{2i+2}w_{8k-2i} & i\equiv 3\pmod{4} 
				\\
			\end{cases}
		\end{equation}
		of $H^*(\BO;\Z/2)$, where by convention all negative Stiefel--Whitney classes are zero.

        \begin{lemma}\label{lem:mathfrakR_in_In}
			In $H^*(\BO;\Z/2)$ we have that $\mathfrak{r}_{2i}\in I_{8k+2} \;(\text{mod }(\mathcal{A}w_2))$ for all $i$.		\end{lemma}
		\begin{proof}
			Using the recursive equation (\ref{eq:RecursiveFormula}) expressing Wu classes in Stiefel-Whitney classes, we calculate that $v_4= w_1^4+w_4 + w_1\Sq^1w_2 + w_2^2  = w_1^4+w_4$ modulo $(\mathcal{A}w_2)$.
			We apply \cref{it:4sqprop} for $2m=8k-2i-2$ and obtain by direct inspection that for any value of $i$ 
			\begin{equation}
				\label{eq:mainRelation}
				\mathfrak{r}_{2i}\equiv \Sq^4(w_1^{2i} w_{8k-2i-2})+v_4w_1^{2i} w_{8k-2i-2} \ \in I_{8k+2} \quad (\text{mod } (\mathcal{A}w_2)).
			\end{equation}
		\end{proof}
		\begin{proof}[Proof of \Cref{prop:8kp2ToPSWClass}]
			We compute
			\begin{align*}
				\sum_{i=0}^{4k-1}\mathfrak{r}_{2i}&=
				\sum_{j=0}^{k-1} (\mathfrak{r}_{8j} +\mathfrak{r}_{8j+2} + \mathfrak{r}_{8j+4} + \mathfrak{r}_{8j+6}) 
				\\
				&=\sum_{j=0}^{k-1}w_1^{8j}(w_{8k+2-8j}+w_1^{8j+4}w_{8k-2-8j}+ w_1^{8j+4}w_{8k-8k-2}+w_1^{8k+6}w_{8k-4-8k}\\&\hspace{10mm} +w_1^{8j+6}w_{8k-8j-4}+w_1^{8j+8}w_{8k-8j-6})\\
				&= \sum_{j=0}^{k-1}(w_1^{8j}w_{8k+2-8j} + w_1^{8j+8}w_{8k-6-8j})
				\\
				&= w_{8k+2} + w_1^{8k} w_{2}
			\end{align*}
			
			By \Cref{lem:mathfrakR_in_In}, the elements $\mathfrak{r_{2i}}\in I_n$ and so $\sum_{i=0}^{4k-1}\mathfrak{r}_{2i}\in I_n$. It follows that $w_{8k+2}\in  I_n \pmod{(\mathcal{A} w_2)}$.
		\end{proof}

		\subsection{\texorpdfstring{$\Pin^-$}{Pin-}-manifolds in dimensions \texorpdfstring{$8k+4$}{8k+4}}
		
		In this section we prove the following.
		\begin{proposition}\label{prop:PinMinusEulerChar} 
			Every $Pin^-$ manifold of dimension $8k+4$ has vanishing top Wu class $v_{4k+2}$ and therefore even Euler characteristic.
		\end{proposition}
		
		We will need the following technical lemma, see also \cite[Prop. 5.1]{ReneeEulerChar}.

		\begin{lemma}\label{lm:SqForPinMinus}
			For every integer $k>0$ we have:
			\[\Sq^{4k+2}=\Sq^2\Sq^{4k}+\Sq^2\Sq^{4k-1}\Sq^{1}.\]
		\end{lemma}
		\begin{proof}
            Recall the Adem relations for Steenrod squares:
		\begin{equation}\label{eq:Adem}
			\Sq^a \Sq^b
			=
			\sum_{j=0}^{\lfloor a/2\rfloor}
			\binom{b-j-1}{a-2j}
			\Sq^{a+b-j}\Sq^j,
			\qquad \text{for } a<2b.
		\end{equation}
		
            In particular we have that
			\[\Sq^{2}\Sq^{4k}={4k-1\choose 2}\Sq^{4k+2}+{4k-2\choose 0}\Sq^{4k+1}\Sq^{1}=\Sq^{4k+2}+\Sq^{4k+1}\Sq^{1},\]
			so we obtain that $\Sq^{4k+2} = \Sq^{2}\Sq^{4k}+\Sq^{4k+1}\Sq^{1}$. By another Adem relation:
			$$\Sq^2\Sq^{4k-1}={4k-2\choose 2}\Sq^{4k+1}+{4k-3\choose 0}\Sq^{4k}\Sq^{1}=\Sq^{4k+1}+\Sq^{4k}\Sq^1,$$ we get $ \Sq^2\Sq^{4k-1}=\Sq^{4k+1}+\Sq^{4k}\Sq^{1}$, which we precompose by $\Sq^1$ to obtain
			
			\[\Sq^2\Sq^{4k-1}\Sq^1=\Sq^{4k+1}\Sq^1+\Sq^{4k}\Sq^{1}\Sq^1.\]
			  A simple application of another Adem relation gives  $Sq^1 Sq^1 = 0$ and we can calculate.
			\[\Sq^{4k+2} = \Sq^{2}\Sq^{4k}+\Sq^{4k+1}\Sq^{1} = \Sq^{2}\Sq^{4k}+\Sq^2\Sq^{4k-1}\Sq^1\]
			which concludes the proof.
			\end{proof}
	
		\begin{proof}[Proof of \Cref{prop:PinMinusEulerChar}]
            A manifold $M^n$ admits a tangential $\xi_{\Pin^{-}}$-structure if and only if the second Wu class $v_2(M)=w_1^2+w_2$ is zero, hence for any $x\in H^{n-2}(M;\Z/2)$, we have
        $    Sq^2(x) = v_2(M) x =0.$
                    For $n=8k+4$ and any $y\in H^{4k+2}(M;\Z/2)$, 
            \[
            v_{4k+2}(M)y = \Sq^{4k+2}(y)= \Sq^2 \left(\Sq^{4k} + \Sq^{4k+1}\Sq^1 \right)(y) = 0,
            \]
using \Cref{lm:SqForPinMinus}, hence it follows that $v_{4k+2}(M)$ vanishes. Finally, since $w_{8k+4}(M)=v_{4k+2}(M)^2$ we get that $w_{8k+4}(M)=0$ for all $\xi_{\Pin^-}$-manifolds $M$ and thus by \Cref{prop:ConditionEvenEuler} all $Pin^-$-manifolds have even Euler characteristic.
        
		\end{proof}

		\subsection{\texorpdfstring{$\Pin^+$}{Pin+}-manifolds in dimension \texorpdfstring{$8k+6$}{8k+6}}
		
		\begin{proposition}\label{prop:pinplusdim6}
			Every $\Pin^+$ manifold of dimension $n=8k+6$ has vanishing top Wu class $v_{4k+3}$ and therefore even Euler characteristic.
		\end{proposition}
		\begin{proof}
			 A manifold $M^n$ admits a tangential $\xi_{\Pin^{+}}$-structure if and only if $w_2(M)=0$. We have $v_3 = w_1 w_2$, so $\xi_{\Pin^+}^* (v_3)=0$, from which it follows that $\Sq^3(x) = 0$ for any $x\in H^{n-3}(M;\Z/2)$.
By an Adem relation (\ref{eq:Adem}), we have for $k\geq 0$
			$$\Sq^{4k+3} = \Sq^3 \Sq^{4k}.$$
Let $n=8k+6$, and $y\in H^{4k+3}(M;\Z/2)$, then
  \[v_{4k+3}(M)y=\Sq^{4k+3}(y)=\Sq^3 \Sq^{4k}(y)=0,\]
  hence $v_{4k+3}=0$ and therefore $w_{8k+3}(M)=v_{4k+3}(M)^2=0$, thus all $(8k+6)$-dimensional $Pin^+$-manifolds have even Euler characteristic.
    \end{proof}

		\section{Application to the SKK-sequence}

		Recall the cut-and-paste groups $\SKK_n^{\xi}$ introduced in \Cref{subsec:SKKIntro}. The $\SKK_n^{\xi}$ groups fit into the following short exact sequence, see  \cite{skbook, gmtw, genauer, KST, HSV2025}.
		
		\begin{theorem}
			\label{thm:TheExactSequence}
			Let $\xi\colon B\to \BO$ be a stable tangential structure.
			Then the canonical map $\SKKxi_n \to \Omega^{\xi}_n$ is well-defined and yields an exact sequence
			\begin{equation} \label{seq:theSesSKK}
				\begin{tikzcd}
					0\ar[r]&\langle S_b^n \rangle_{\SKKxi_n} \ar[r,"s"]&\SKK_n^{\xi}\ar[r]& \Omega^\xi_n \ar[r]& 0,
				\end{tikzcd}
			\end{equation}
			where $\langle S^n_b\rangle$ is a cyclic group and the map $s$ sends a generator to the sphere $S^n$ with the unique $\xi$-structure that extends to the disk $D^{n+1}$.
			
			If $n$ is even $\langle S_b^n \rangle_{\SKKxi_n} \cong \Z$, and if $n$ is odd then
			\begin{enumerate}[label=(\roman*)]
				\item \label{it:ses1}$[S^n_b] = 0 \in \SKKxi_n$ if there exists a closed $(n+1)$-dimensional $\xi$-manifold with odd Euler characteristic;
				\item\label{it:ses2}$\langle S_b^n \rangle_{\SKKxi_n} \cong \Z/2$ if all closed $(n+1)$-dimensional $\xi$-manifolds have even Euler characteristic.
			\end{enumerate}
		\end{theorem}

		So we obtain
		\begin{corollary}\label{cor:SKKPin}
			 Let $n\geq 1$ be an odd integer and let $\xi$ be $B\Pin^{+}\to \BO$ or $B\Pin^{-}\to \BO$. Then we get that 
			\begin{eqnarray*}
			\SKK_n^{Pin^-} &\cong&     \begin{cases}
            \Omega_{n}^{\Pin^-}	&\text{ if } n=1,5,7 \pmod{8}\\
				\Omega_{n}^{\Pin^-}\oplus\Z/2& \text{ if }  n=3 \pmod{8}
			\end{cases}\\
\SKK_n^{Pin^+} &\cong& \begin{cases}
				\Omega_{n}^{\Pin^+}&\text{ if }n=3,7 \pmod{8}\\
				\Omega_{n}^{\Pin^+}\oplus\Z/2& \text{ if }n=5 \pmod{8} 
\end{cases}
            \end{eqnarray*}
            and lastly for $\Pin^+$-manifolds in dimension $n= 1 \pmod{8}$ there is a short exact sequence $\Z/2\to\SKK_n^{\Pin^+}\to \Omega_n^{\Pin^+}$.
    
            In the above the $\Z/2$ summands are generated by the spheres $S^n_b$ with $\xi$-structures which extend to the $n+1$-disc. The maps $\SKK_n^{\xi}\to \Omega_n^{\xi}$ are given by the inclusion and in the split cases the map $\SKK^{\xi}_n\to \Z/2$ is given by the $\Z/2$-Kervaire semi-characteristic
            \[\kerv_{\Z/2}(M^{2k+1})=\sum_{i=0}^{k}\rank(H_i(M;\Z/2)).\]
		\end{corollary}
		\begin{proof}
			That $\SKK^{\xi}_n$ is isomorphic to either a bordism group or to an extension of the bordism group by $\Z/2$ follows from \Cref{mnthm:Mainthm} and \Cref{thm:TheExactSequence}. The fact that in the $8k+5$ dimensional $\Pin^+$ case and the $8k+3$ dimensional $\Pin^-$ case is split by the Kervaire semi-characteristic follows from \cite[Theorem B]{HSV2025} since in \Cref{prop:pinplusdim6} and \Cref{prop:PinMinusEulerChar} we show that in those cases the top Wu classes vanish.
		\end{proof}

        \begin{lemma}\label{lem:Pin+counterexamples}
            In every dimension $n=8k+2$ there exists a $\Pin^+$ manifold $M^{8k+2}$ such that $v_{4k+1}(M)\neq 0$, even though $w_{8k+2}(M)=0$. 
        \end{lemma}
        \begin{proof}
            Consider the manifold  $M:=K^2\times \HP^{2k}$ where $K^2$ is the Klein bottle. Now $w_2(M)= w_2(K)+ w_1(K) w_1(\HP^{2k}) + w_2(\HP^{2k}) = w_2(K)$ which vanishes since $w_2(K)=\chi(K)\pmod{2}=0$. Now $v_{4k+1}(M)=  v_1(K)v_{4k}(\HP^{2k})$. Since $\chi(\HP^{2k})$ is nonzero, $w_{8k}(\HP^{2k})\neq 0$ and so $v_{4k}(\HP)\neq 0$. Furthermore $w_1(K)=v_1(K)\neq 0$ which proves  the claim.
        \end{proof}
         Consider \cite[Proposition 4.20]{HSV2025} which says that if the SKK sequence (\ref{seq:theSesSKK}) in the dimension $8k+1$ for $\Pin^+$ is split, then the splitting has to non-trivially depend on the $\Pin^+$-structures (in other words, no purely topological invariant $\SKK_{8k+1}^{\Pin^+}\to \Z/2$ can be a splitting). It follows from this fact and \Cref{lem:Pin+counterexamples} that the $\Z/2$-Kervaire semi-characteristic {\it does not} split the SKK-sequence (\ref{seq:theSesSKK}). Thus we can ask:
\begin{question}
   Does the exact sequence
    \begin{equation}\label{eq:Pinskk}
            0 \to \Z/2 \to \SKK^{\Pin^+}_{8k+1} \to \Omega^{\Pin^+}_{8k+1} \to 0
        \end{equation}

    split for all $k \geq 0$? 
\end{question}
The sequence splits trivially for $k=0,1$ since $\Omega^{\Pin^+}_{1}=0$ and $\Omega^{\Pin^+}_{9}=0$ \cite{kirbytaylor}.

\begin{lemma}
    If there exists a $k$ for which \cref{eq:Pinskk} does not split, then for any $k'\geq k$, it also does not split.
\end{lemma}
\begin{proof}
    If for an integer $k$  \cref{eq:Pinskk}
    is not split then there exists a manifold $X^{8k+1}$ with the property that $2[X]=[S_b^{8k+1}]$ in $\SKK^{\Pin^+}_{8k+1}$. Let $k'>k$. Consider the manifold $X'= X \times \mathbb{HP}^{2(k'-k)}$. Let $W$ be a nullbordism of $X \sqcup X$. Using \cite[Lemma 3.13]{HSV2025} we have that in $\SKK^{\Pin^+}_{8k+1}$
    \[2[X]=\chi(W)[S^{8k+1}]\]
    and so $\chi(W)\equiv 1\pmod{2}$. Then $W\times \HP^{2(k'-k)}$ is a nullbordism of $X'\sqcup X'$ and using the same result we have that in $\SKK^{\Pin^+}_{8k'+1}$
    \[
    [X' \sqcup X'] = \chi(W)\chi(\mathbb{HP}^{2(k'-k)})[S_b^{8k'+1}] \equiv [S_b^{8k'+1}],
    \]
    which proves the result.
\end{proof}

		\bibliography{biblio.bib}{}
		\bibliographystyle{amsalpha}
		
	\end{document}